\documentclass{article}
\usepackage{amsthm,amssymb,amsmath,enumerate,graphicx}
\pdfoutput=1 

\newtheorem{THM}{Theorem}[section]
\newtheorem{LEM}[THM]{Lemma}

\newtheorem{PROP}[THM]{Proposition}

\newtheorem{EX}[THM]{Example}
\newtheorem{PROBLEM}[THM]{Problem}

\def\ord(#1){|#1|}
\def\bigord(#1){\big|#1\big|}

\newcommand\abs[1]{\lvert #1\rvert}

\def\shift(#1)(#2){\!\!\downarrow^{(#1)}_{(#2)}} 
\def\ucl(#1){\lfloor #1 \rfloor}
\def\dcl(#1){\lceil #1 \rceil}

\newcommand\B{\mathcal B}

\def\F{\mathcal F}
\renewcommand\H{\mathcal H}
\renewcommand\P{\mathcal P}
\def\S{\mathcal S}

\def\sub{\subseteq}
\def\supe{\supseteq}
\def\sm{\smallsetminus}
\def\td{tree-decom\-po\-si\-tion}

\newcommand\COMMENT[1]{}

\def\?#1{\vadjust{\vbox to 0pt{\vss\vskip-8pt\leftline{%
     \llap{\hbox{\vbox{\pretolerance=-1
     \doublehyphendemerits=0\finalhyphendemerits=0
     \hsize16truemm\tolerance=10000\small
     \lineskip=0pt\lineskiplimit=0pt
     \rightskip=0pt plus16truemm\baselineskip8pt\noindent
     \hskip0pt        
     #1\endgraf}\hskip7truemm}}}\vss}}}

\title{Unifying duality theorems for width parameters in graphs and matroids\smallskip\\ II.~General duality}
\author{Reinhard Diestel \and Sang-il Oum%
\thanks{The second-named author was supported by Basic Science Research
  Program through the National Research Foundation of Korea (NRF)
  funded by  the Ministry of Science, ICT \& Future Planning
  (2011-0011653).}
 }

\begin{document}
\abovedisplayshortskip=-3pt plus3pt
\belowdisplayshortskip=6pt

\maketitle

\begin{abstract}\noindent
  We prove a general duality theorem for tangle-like dense objects in combinatorial structures such as graphs and matroids. This paper continues, and assumes familiarity with, the theory developed in~\cite{DiestelOumDualityI}.
\end{abstract}

\section{Introduction}\label{sec:intro}

This is the second of two papers on the duality between certain `dense objects' in a combinatorial structure such as a graph or a matroid, and tree-like structures to which any graph or matroid that does not contain such a `dense object' must conform. Its `tree-likeness' is then measured by a so-called \emph{width parameter}, which is the smaller the more the graph or matroid conforms to such a tree-shape. The `dense objects', which traditionally come in various guises, are in our framework cast uniformly in a way akin to tangles: as orientations of certain separation systems of the graphs.

In Part~I of this paper~\cite{DiestelOumDualityI} we proved a duality theorem which unifies and extends the duality theorems for the classical width parameters of graphs and matroids, such as branch-width, tree-width and path-width. Our aim, however, had not been to find such a theorem. What we were looking for was a duality theorem for more general width parameters still, one that would also cover more recently studied tangle-like objects such as `blocks'~\cite{confing, ForcingBlocks} and `profiles'~\cite{profiles, CDHH13CanonicalAlg, CDHH13CanonicalParts}.

Amusingly, our theorem from~\cite{DiestelOumDualityI} covers neither blocks nor profiles, although it does cover all those other parameters, and blocks are similar to tangles while profiles are sandwiched between tangles and brambles (and generalize blocks). We shall see in this paper why our attempt had to fail: we prove that both blocks and profiles need more general obstructions to witness their non-existence than the structure trees used in~\cite{DiestelOumDualityI}.

As our main positive result, we shall prove a duality theorem that does cover such general tangle-like objects as blocks and profiles (and countless others). Since the obstructions it identifies for their non-existence are more general than trees, it will not imply our results from~\cite{DiestelOumDualityI} nor follow from them.

We shall use the same terminology as in~\cite{DiestelOumDualityI}; see Section~\ref{sec:background} for what exactly we shall assume the reader is familiar with. In Section~\ref{sec:stars} we describe the general tangle-like structures we shall cover, while in Sections~\ref{sec:obstructions}--\ref{sec:Sgraphs} we describe the more general tree-like structures to witness their non-existence. Our General Duality Theorem will be proved in Section~\ref{sec:Thm}. In Section~\ref{sec:apps} we apply it to blocks and pro\-files (which are also defined formally there), and show that for these structures our duality theorem is best possible: there are graphs that have neither a block or profile nor admit a tree-like decomposition as in the duality theorem of~\cite{DiestelOumDualityI}, thus showing that the more general tree-like structures employed by our main result are best-possible not only for a general duality theorem but also for just blocks and profiles.

\section{Background needed for this paper}\label{sec:background}

We assume that the reader is familiar with the terminology set up in Part~I\penalty-200\ \cite[Section~2]{DiestelOumDualityI}, and with the proof of its Weak Duality Theorem~\cite[Section~3]{DiestelOumDualityI}. Let us restate this theorem:

\begin{THM}[Weak Duality Theorem]\label{thm:weak}
  Let ${S}$ be a separation system of a set~$V$, and let ${S}^-\sub S$ contain every separation~of the form $(A,V)\in \mathcal S$. Let $\F$ be a set of stars in ${S}$. Then exactly one of the following holds:
\begin{enumerate}[\rm(i)]\itemsep0pt\vskip-3pt\vskip0pt
  \item There exists an ${S}$-tree over $\F$ rooted in ${S}^-$.
  \item There exists an $\F$-avoiding orientation of ${S}$ extending ${S}^-$.
  \end{enumerate}
\end{THM}

\noindent
The conclusion of the Strong Duality Theorem~\cite{DiestelOumDualityI} differs from this in one crucial detail: it asks that the orientations in (ii) be consistent, that they never orient two separations away from each other.%
   \footnote{\dots which indeed makes little sense if we think of these orientations as pointing to some dense object.}
  This comes at the price of having to restrict $S$ and~$\F$ in the premise, but we shall not need the technicalities of this restriction in this paper.

As pointed out earlier, our motivation for starting this line of research was to find a duality theorem, and associated width parameter, for two notions of `dense objects' that have recently received some attention, so-called \emph{blocks} and \emph{profiles} (defined at the start of Section~\ref{sec:apps}).

As it turns out, these cannot be captured in the framework developed in~\cite{DiestelOumDualityI}, as (consistent) orientations of separation systems avoiding a certain collection~$\F$ of stars of separations. Our General Duality Theorem will therefore relax this requirement by allowing $\F$ to contain arbitrary `forbidden' sets of separations. However, we shall see in Section~\ref{sec:stars} that it will suffice to consider so-called `weak stars' as elements of~$\F$.

In Sections \ref{sec:obstructions}--\ref{sec:Sgraphs} we deal with the other side of the duality, the tree-structure. It turns out that this, too, has to be relaxed for any duality theorem for blocks or profiles. Of course, we would have to allow $S$-trees over weak stars rather than just over stars as in~\cite{DiestelOumDualityI}, but even this is not enough: we need to relax the trees to certain graphs with few cycles, which we shall call \emph{$S$-graphs}.

\section{From stars to weak stars}\label{sec:stars}

Recall that a set $S = \{\,(A_i,B_i) : i=1,\dots,n\,\}$ of separations of a set $V$ is a \emph{star} if these separations are nested and point towards each other, that is, if $(A_i,B_i)\le (B_j,A_j)$ for all distinct $i,j\le n$. In both the Weak and the Strong Duality Theorem proved in~\cite{DiestelOumDualityI}, the sets of separations that were forbidden in the orientations of a separation system defining a particular `dense object', those in~$\F$, were stars.

In the General Duality Theorem we shall prove here, $\F$~can be an arbitrary collection of sets of separations. It will be good, however, to be able to restrict this arbitrariness if desired: the smaller we can make~$\F$, the easier will it be to show that an orientation is $\F$-avoiding.%
   \COMMENT{}
   In this section we show that $\F$ can always be restricted to the `weak stars' it contains.%
   \COMMENT{}

   \begin{figure}[htpb]
\centering
   	  \includegraphics{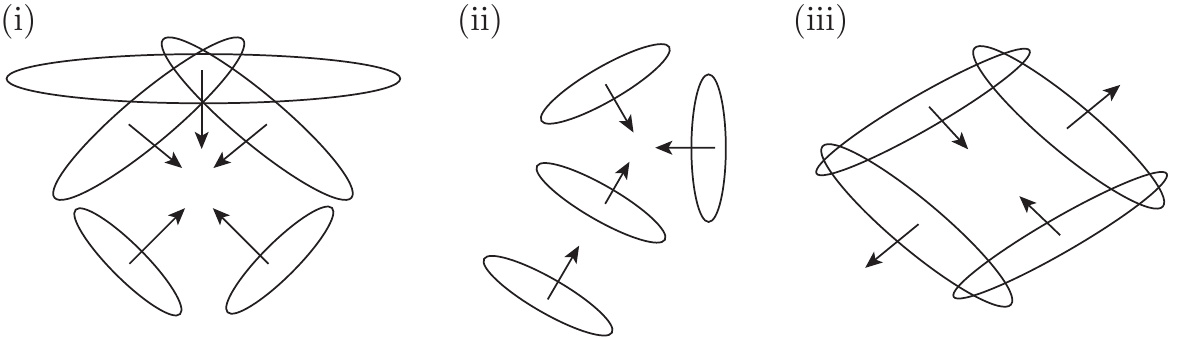}
   	  \caption{The separations in (i) form a weak star; those in (ii) and (iii) do not.}
   \label{fig:weakstar}
   \end{figure}

A set of separations is a \emph{weak star} if it is a consistent $\le$-antichain (Fig.~\ref{fig:weakstar}). Clearly, antisymmetric%
\footnote{A set $S$ of separations is \emph{antisymmetric} if $(B,A)\notin S$ whenever $(A,B)\in S$.}
stars of proper separations%
   \COMMENT{}
   are weak stars, and a weak star is a star if and only if it is nested. Given a set $S$ of separations, let $S^+\sub S$ denote the set of its $\le$-maximal elements. Then, for any collection $\F$ of sets of separations, all the elements of
  $$\F^* := \{\, S^+\! : S\in\F\text{ and $S$ is consistent}\,\}$$
 are weak stars. Note that, formally, $\F^*$~need not be a subset of~$\F$. But in all our applications it will be,%
   \COMMENT{}
   and if it is, it is just the subset of $\F$ consisting of all its weak stars.%
   \COMMENT{}

Given two sets $S,S'$ of separations of~$V\!$,  let us define
 $$S\le S'\ :\Leftrightarrow\ \exists f\colon S\to S'\text{ such that } (A,B)\le f(A,B)\text{ for all }(A,B)\in S.$$
This is a reflexive and transitive relation on the sets of separations of~$V\!$, and on the weak stars it is also antisymmetric.%
   \COMMENT{}
   Given any collection $\F$ of sets of separations, we let
 $$\F^- :=\, \{\, S : S \text{ is $\le$-minimal in } \F^*\}.$$

\begin{LEM}\label{lem:stars}
The following statements are equivalent for every consistent orientation $O$ of a finite%
   \COMMENT{}
   separation system $R$ of~$V\!$ and $\F\sub 2^R$.
\begin{enumerate}[\rm(i)]\itemsep=0pt\vskip-3pt\vskip0pt
   \item $O$ avoids~$\F$.
   \item $O$ avoids~$\F^*$.
   \item $O$ avoids~$\F^-$.
\end{enumerate}
\end{LEM}

\begin{proof}
(i)$\to$(ii)$\to$(iii): Consider any set $S\in\F$. If $S^+\sub O$ then also $S\sub O$, because $O$ is a consistent orientation of $R\supe S$, and every separation in~$S$ lies below $(\le)$ some separation in~$S^+$. Hence if $O$ avoids~$\F$ it also avoids~$\F^*$, and thus also $\F^-\sub\F^*$.

(iii)$\to$(ii):%
   \COMMENT{}
   If $S^-\le S\sub O$ then $S^-\sub O$, since $O$ is a consistent orientation of $R\supe S^-$. Since for every $S\in\F^*$ there exists an $S^-\le S$ in~$\F^-$, this proves the assertion.

(ii)$\to$(i): If $O$ has a subset $S$ in~$\F$, then $S\sub O$ is consistent, so $S^+\sub S\sub O$ lies in~$\F^*$. Hence if $O$ avoids~$\F^*$, it must avoid~$\F$.
\end{proof}

\section{Obstructions to consistency}\label{sec:obstructions}

As pointed out earlier, the key advance of our Strong Duality Theorem over the weak one is that the orientations of the given separation system $S$ whose existence it claims will be consistent. This comes at a price: to obtain consistent orientations in the proof we had to impose a condition on $S$ and~$\F$, that $S$ should be `$\F$-separable'.%
   \COMMENT{}

The following simple example shows that imposing some condition was indeed necessary: even if $\F$ consists only of stars, it can happen that there is neither a consistent $\F$-avoiding orientation of~$S$ nor an $S$-tree over~$\F$ that extends some given consistent $S^-\sub S$.

\begin{EX}\label{ex:P1}\rm
Let $S$ consist of two pairs of crossing separations and their inverses: $(A_1,B_1), (A_2,B_2)$ and $(A'_1,B'_1), (A'_2,B'_2)$, such that $(A_i,B_i) < (A'_j,B'_j)$ for all choices of $\{i,j\} \subseteq \{1,2\}$. Let $S^- = \{(A_1,B_1), (B'_1,A'_1)\}$, and let $\F$ contain the stars $S_1 = \{(A_1,B_1), (B'_2,A'_2)\}$ and $S_2 = \{(A_2,B_2), (B'_1,A'_1)\}$ (Fig.~\ref{fig:P1graph}).

There exists a unique $\F$-avoiding orientation of $S$ extending~$S^-$, which contains $(A'_2,B'_2)$ to avoid $S_1$ in the presence of~$(A_1,B_1)\in S^-$, as well as $(B_2,A_2)$ to avoid $S_2$ in the presence of $(B'_1,A'_1)\in S^-$. But these two separations, $(A'_2,B'_2)$ and $(B_2,A_2)$, face away from each other, so this orientation of $S$ is inconsistent.

However, there is no $S$-tree over~$\F$ rooted in~$S^-$. Indeed, the only star in $\F$ containing $(A_1,B_1)\in S^-$ is~$S_1$, the only star in $\F$ containing $(B'_1,A'_1)\in S^-$ is~$S_2$, and there are no further stars in~$\F$. Figure~\ref{fig:P1graph} shows the forest $F$ that arises instead of an $S$-tree on the right.
\end{EX}

   \begin{figure}[htpb]
\centering
   	  \includegraphics
           {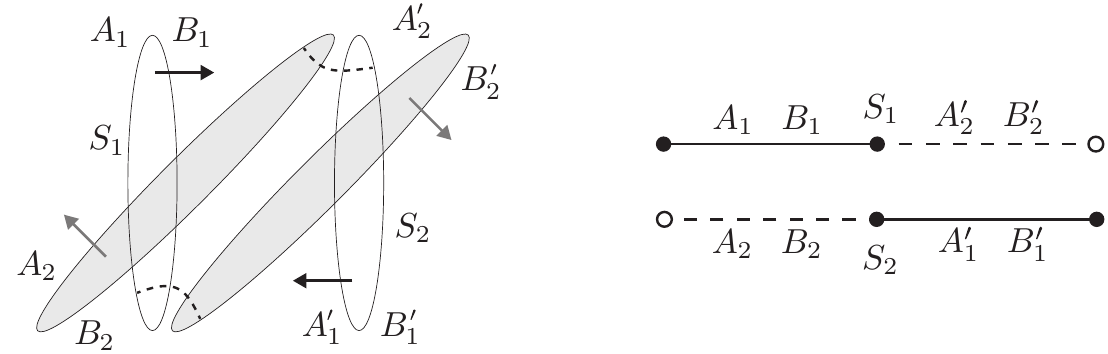}
   	  \caption{A separation system $S$ that has no consistent orientation extending~$S^-\!$, but also no $S$-tree over~$\F$ rooted in~$S^-$. Its `$S$-forest' is shown on the right.}
   \label{fig:P1graph}
   \end{figure}

However, we can repair the `$S$-forest'~$F$ by generalizing the notion of an $S$-tree, as follows. The idea of an $S$-tree is that it should witness the non-existence of certain orientations of~$S$. The non-existence of a \emph{consistent} orientation of~$S$ should be easier to witness, because it is a weaker property. And indeed, we can endow our $S$-trees with an additional feature to witness violations of consistency: nodes of degree~2 whose incoming edges map to separations that point \emph{away from} each other.%
   \footnote{This is the opposite of requiring the incoming edges at a node to map to a star of separations, which would have these point towards each other. In particular, if we allow $(T,\alpha)$ such new nodes of degree~2, then $\alpha$ will no longer preserve the natural orientation of~$\vec E(T)$.}
   As the reader may check, a tree $(T,\alpha)$ whose oriented edges map to separations in such a way that (incoming) stars at nodes either map to stars in $\F$ and or are 2-stars of this new type will still witness the non-existence of a consistent $S$-orientation avoiding~$\F$ (where, as before, leaf separations must be in $S^-$, which in turn must be included in any $S$-orientation considered): any consistent $\F$-avoiding orientation $O$ of $S$ will induce, via~$\alpha$, an orientation of $\vec E(T)$ in which no interior node of~$T$ is a sink; so there has to be a sink at a leaf, implying $O\not\supe S^-$ if $(T,\alpha)$ is rooted in~$S^-$.

   \begin{figure}[htpb]
\centering
   	  \includegraphics
             {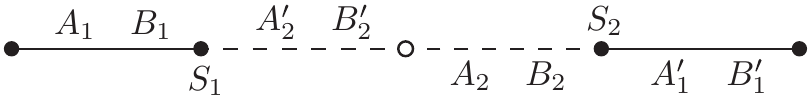}
   	  \caption{Constructing a generalized $S$-tree}
   \label{fig:P1tree}
   \end{figure}

Figure~\ref{fig:P1tree} shows such a `generalized $S$-tree'. We shall build on this idea later when we define `$S$-graphs' over arbitrary sets $\F\sub 2^S$, the objects dual to consistent $\F$-avoiding $S$-orien\-ta\-tions in our General Duality Theorem.

Alternatively, one might suspect that the failure of duality in Example~\ref{ex:P1} stems from our unhelpful choice of~$S^-$. Our next example of a separation system $S$ has no consistent orientation whatsoever avoiding a certain $\F\sub 2^S$ (again consisting of stars), none extending even $S^-\!=\emptyset$. But as there are no trees without leaves, it cannot have an $S$-tree rooted in~$\emptyset$ either, not even one generalized as above.

\begin{EX}\label{ex:P1cycle}\rm
Let $S$ consist of five inverse pairs $S_1,\dots,S_5$ of separations, arranged cyclically so that each crosses its predecessor and its successor but is nested with the other two pairs of separations, as shown in Fig.~\ref{fig:P1cycle}.%
   \COMMENT{}
   Let $S^-=\emptyset$. Let $\F$ consist of all the stars in~$S$, each consisting of two separations pointing towards each other. Then in any consistent $\F$-avoiding orientation $O$ of~$S$ every two nested separations will be comparable: they will not point towards each other because $O$ avoids~$\F$, and they will not point away from each other because $O$ is consistent. We shall prove that $S$ has no such orientation~$O$. Since $S^-=\emptyset$, it cannot have an $S$-tree either, not even one generalized as above.

For each $i=1,\dots,5$, let $\vec S_i$ be the separation from the pair~$S_i$ that lies in~$O$. Let $H$ be the graph on $\{\vec S_1,\dots,\vec S_5\}$ in which two $\vec S_i$ form an edge whenever they are nested. This is a 5-cycle; pick one of its two orientations. Since nested separations in $O$ are comparable, we have for each oriented edge $(\vec S_i, \vec S_j)$ in $H$ either $\vec S_i < \vec S_j$, in which case we colour the edge green, or $\vec S_i > \vec S_j$, in which case we colour it red. Since 5 is odd, $H$~has two equally coloured edges, so $O$ contains three adjacent separations $\vec S_i < \vec S_j < \vec S_k$. But there are no three pairwise nested separations in~$S$, a contradiction.

Figure~\ref{fig:P1cycle}, right, shows an attempted `generalized $S$-tree'~-- in fact, a 10-cycle~-- pieced together by (incoming) 2-stars mapping to~$\F$ (solid nodes) or, alternately, witnessing the inconsistency of two separations (hollow vertices).
\end{EX}

   \begin{figure}[htpb]
\centering
   	  \includegraphics
              {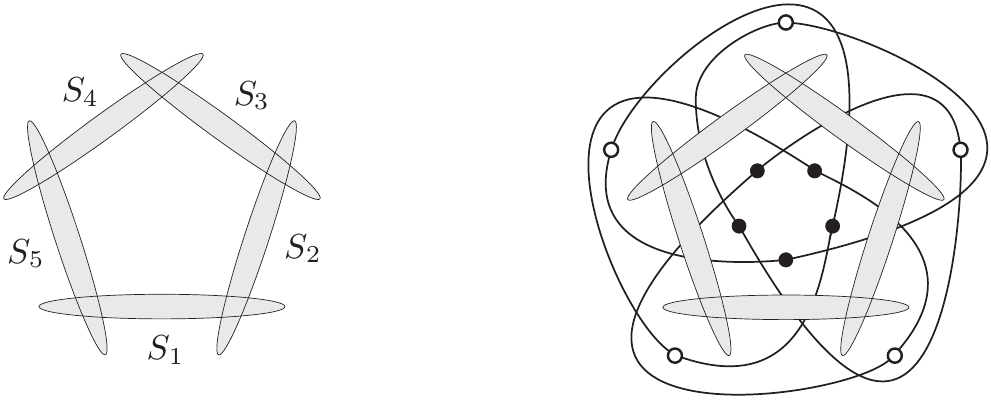}
   	  \caption{Cyclically arranged separations not defining a generalized $S$-tree}
   \label{fig:P1cycle}
   \end{figure}

Note that our proof of the non-existence of a consistent $\F$-avoiding orientation in Example~\ref{ex:P1cycle} was quite ad-hoc and perhaps not obvious. As an application of our General Duality Theorem, we shall later see a more canonical proof, illustrated by Figure~\ref{fig:5cycleSgraph}.

\section{From $\boldsymbol S$-trees to $\boldsymbol S$-graphs}\label{sec:Sgraphs}

The essence of our duality theorems is that if a given separation system $S$ has no $\F$-avoiding orientation (possibly consistent) that extends a given set $S^-\sub S$, then this is witnessed by a particularly simple subset of~$S$: a nested subsystem that already admits no $\F$-avoiding orientation extending~$S^-$. Since nested separation systems define tree-like decompositions of the structures they separate, it is convenient to describe them as $S$-trees.

As we saw in Section~\ref{sec:obstructions}, however, it can happen for certain choices of $S$ and~$\F$ that $S$ admits no consistent $\F$-avoiding orientation but every nested subsystem of $S$ does.%
   \COMMENT{}
   In such cases, it may still be possible to find a subsystem of $S$ which, though perhaps not nested, is still considerably simpler than $S$ and also admits no consistent $\F$-avoiding orientation, and which can thus be used as a witness to the fact that $S$ admits no such orientation.

It is our aim in this section to present a structure type for separation subsystems, slightly more general than nested systems, that can always achieve this. These structures are formalized as \emph{$S$-graphs}, a~generalization of $S$-trees just weak enough to describe such systems when they are not nested. Both $S$-trees and the `generalized $S$-trees' considered in Section~\ref{sec:obstructions} will be examples of $S$-graphs.

In order to illuminate the idea behind their definition, let us recall the standard proof of why any $S$-tree $(T,\alpha)$ over $\F$ rooted in~$S^-$ is an obstruction to $\F$-avoiding orientations $O\supe S^-$. Via~$\alpha$, the orientation $O$ of~$S$ orients the edges of~$T$. Since the stars at nodes of $T$ map to stars in~$\F$, which are never subsets of~$O$, the edges of $T$ at a given node are never all oriented towards it. Similarly, since $(T,\alpha)$ is rooted at~$S^-\sub O$, no edge of $T$ is oriented towards a leaf. Hence $T$ has no sink in this orientation, which cannot happen.

In this proof we did not use that $T$ is a tree except at the end, when we needed that no orientation of $T$ can leave it without a sink. And neither did we use that the separations in the image of $\alpha$ are nested. An $S$-graph, in the same spirit, will be a graph $H$ with a map from its edge orientations to~$S$ such that any consistent $\F$-avoiding orientation of $S$ will orient the edges of $H$ in a way that contradicts its structure.

The following types of graph will be used as $S$-graphs. Consider finite connected bipartite undirected graphs $H$ with at least one edge and bipartition $V(H) = N\cup M$. For every vertex $m\in M$ let its set $E(m)$ of incident edges be partitioned into two non-empty classes, as $E(m) = E'(m)\cup E''(m)$. Call every vertex of degree~1 a \emph{leaf}, and assume that all leaves lie in~$N$. Let $\H$ be the class of all such graphs.

Let $S$ be a separation system of a set~$V$, and let $S^-\sub S$ and $\F\sub 2^S$. An \emph{$S$-graph} over $\F$ rooted in~$S^-$ is a pair $(H,\alpha)$ such that $H$ is a graph in~$\H$, with bipartition $N\cup M$ say, and $\alpha\colon \vec E(H)\to S$ satisfies the following:
   \begin{enumerate}[(i)]
   \item $\alpha$ commutes with inversions of edges in $\vec E(H)$ and of separations in~$S$, i.e., $\alpha(u,v) = (A,B)$ implies $\alpha(v,u) = (B,A)$;
   \item the incident edge $nm$ of any leaf $n\in N$ satisfies $\alpha(n,m)\in S^-$;
   \item for every node $n\in N$ that is not a leaf, with incident edges $nm_1,\dots,nm_k$ say, $\{\alpha(m_1 n),\dots,\alpha(m_k n)\}\in\F$;
   \item $\alpha(n',m)\ge\alpha(m,n'')$ whenever $m\in M$ and $n'm\in E'(m)$, $mn''\in E''(m)$.
   \end{enumerate}

Note that (i)--(iii) are copied from the definition of an $S$-tree over~$\F$ rooted in~$S^-$. In particular, (iii) makes $\alpha$ map oriented stars $\{(m_1,n),\dots,(m_k n)\}$ at nodes $n\in N$ to `forbidden' sets of separations in~$\F$. Similarly, (iv) makes $\alpha$ map oriented stars $\{(n',m),(n'',m)\}$ at vertices $m\in M$ to pairs of separations that violate consistency (Fig.~\ref{fig:m}). We shall refer to (ii) by saying that $(H,\alpha)$ is \emph{rooted in~$S^-$}, to (iii) by saying that $(H,\alpha)$ is \emph{over~$\F$}, and to (iv) by saying that the edges of $H$ at vertices $m\in M$ \emph{witness inconsistencies}.

   \begin{figure}[htpb]
\centering
   	  \includegraphics{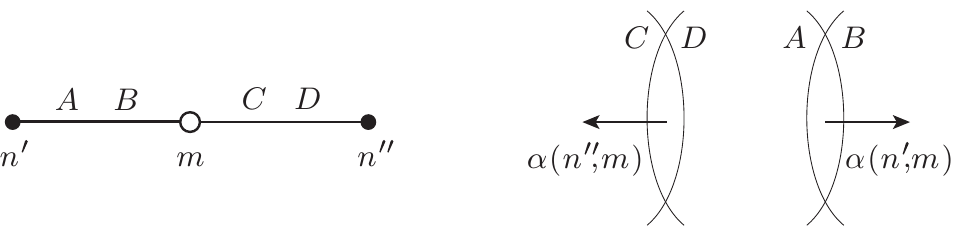}
   	  \caption{At~$m$, the map $\alpha$ witnesses the inconsistency of $(A,B) = \alpha(n',m')$ with $(D,C) = \alpha(n'',m)$, as $(A,B)\ge (C,D)$.}
   \label{fig:m}
   \end{figure}

Thus, any $S$-tree $(T,\alpha')$ over $\F$ rooted in~$S^-$ becomes an $S$-graph $(H,\alpha)$ over $\F$ rooted in~$S^-$ on subdividing every edge $n'n''$ by a vertex~$m$ and letting $\alpha(n',m) = \alpha(m,n'') = \alpha'(n',n'')$. And so do the `generalized $S$-trees' considered in Section~\ref{sec:obstructions} if we subdivide the original edges of their underlying $S$-tree.\looseness=-1%
   \footnote{Not their new edges witnessing violations of consistency, which already come in pairs sharing a vertex $m$ of degree~2 and satisfy~(iv).}

Note also that condition (iv) is invariant under swapping the names of $E'(m)$ and~$E''(m)$, since $\alpha(n',m)\ge \alpha(m,n'')$ implies $\alpha(n'',m)\ge\alpha(m,n')$ by~(i). The only purpose of partitioning $E(m)$ into these two sets is to be able to define `traversing' below; it does not matter which of the two partition sets is~$E'(m)$ and which is~$E''(m)$.

Finally, we remark that condition (ii) could be subsumed under~(iii) by putting the inverses of separations in $S^-$ in $\F$ as singleton sets and applying (iii) also to leaf nodes~$n$: since any orientation of $S$ must include every separation from $S^-$ or its inverse, forbidding the inverses of separations in~$S^-$ amounts to including~$S^-$. However, it will be convenient in the proof of our duality theorem to treat the two separately.%
   \COMMENT{}

The $S$-graphs we shall in fact need will have some further properties that make them more like trees.
Let us say that a path or cycle in a graph $H\in\H$ as above \emph{traverses} $m\in M$ if it has an edge in $E'(m)$ and another in~$E''(m)$. We shall call~$H$, and any $S$-graph $(H,\alpha)$ based on it, \emph{cusped}%
   \footnote{The word `cusped' is intended to convey a notion of `nearly as spiky as a tree': by (ii) below, any cycle in a cusped graph must have a `cusp' at a vertex $m\in M$, entering and leaving it through the same partition class of~$E(m)$.}
 if
   \begin{enumerate}[(i)]
   \item every edge $mn$ such that $n\in N$ is a leaf is the only edge in its bipartition class of~$E(m)$;
   \item no cycle in $H$ traverses all the vertices of $M$ that it contains.
   \end{enumerate}
Note that non-trivial trees become cusped graphs if we subdivide every edge.%
   \COMMENT{}

\goodbreak

We shall prove that cusped $S$-graphs over $\F$ that are rooted in~$S^-$ are obstructions to consistent $\F$-avoiding orientations extending~$S^-$.%
   \COMMENT{}
    The following property of cusped graphs is at the heart of that proof:

\begin{LEM}\label{lem:pointedsink}
Let $H$ be a cusped graph, with bipartition classes $M,N$ and $E(m) = E'(m)\cup E''(m)$ for all $m\in M$. For every orientation of its edges, $H$~either has a node $n\in N$ with all incident edges oriented towards~$n$, or it has a vertex $m\in M$ such that both $E'(m)$ and~$E''(m)$ contain an edge oriented towards~$m$.
\end{LEM}

\begin{proof}
In a given orientation of~$H$, let $Q$ be a maximal forward-oriented path that starts at a node in~$N$ and traverses every $m\in M$ it contains unless it ends there.

Suppose first that $Q$ ends at a vertex $m\in M$. Then $Q$ has an edge in only one of the two partition classes of~$E(m)$. If all the edges in the other partition class are oriented towards~$m$, then $m$ has the desired property, because that other partition class is also non-empty (by definition of~$\H$). If not, the other partition class contains an edge $mn$ oriented towards~$n$ and with $n\in Q$. Adding this edge to the final segment $nQ$ of~$Q$ we obtain a cycle in $H$ that traverses all its vertices in~$m$, a contradiction.

Suppose now that $Q$ ends at a node $n\in N$. If all the edges of $H$ at $n$ are oriented towards~$n$, then $n$ is as desired. If not, there is an edge $nm$ oriented away from~$n$. By the maximality of $Q$ we cannot append this edge to~$Q$, so $Q$ traverses~$m$. Since the cycle obtained by adding the edge $nm$ to the final segment $mQ$ of $Q$ does not traverse all its vertices in~$M$,%
   \COMMENT{}
   the first edge of $mQ$ lies in the same partition class as~$nm$. Then the edge preceding $m$ on~$Q$ and the edge $nm$ are both oriented towards~$m$ and lie in different partition classes of~$E(m)$, as desired.
\end{proof}

We remark that the converse of Lemma~\ref{lem:pointedsink} can fail: the graph shown in Figure~\ref{fig:pointedsinkconverse}, in which the two hollow vertices are in $M$ and their incident edges are partitioned into `left' and `right', satisfies the conclusion of the lemma but contains a cycle that traverses all its vertices in~$M$.

  \begin{figure}[htpb]
\centering
   	  \includegraphics
           {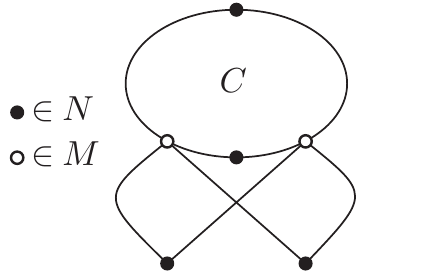}
   	  \caption{A counterexample to the converse of Lemma~\ref{lem:pointedsink}}
   \label{fig:pointedsinkconverse}
   \end{figure}

Although we shall prove that any cusped $S$-graph can be used as a witness to the non-existence of the corresponding orientations of~$S$, we shall also prove that there are always witnesses among these that can be constructed in a particularly simple way:\vadjust{\penalty-200} recursively from subdivided $S$-trees by amalgamations in a single vertex. Let us call a graph $H$ in $\H$ \emph{constructible} if either
   \begin{enumerate}[(P1)]
   \item $H$ is obtained from a $k$-star on~$N$ with $k\ge 2$ by subdividing every edge once and putting the subdividing vertices in~$M$;%
      \COMMENT{}
      or
   \item $H$ is obtained from the disjoint union of two constructible graphs $H',H''$ as follows. Let $V(H') = N'\cup M'$ and $V(H'') = N''\cup M''$ with the familiar notation. Let $L'$ be a non-empty set of leaves of~$H'$ such that, for every $n'\in L'$, its incident edge~$n'm'$ is the unique edge in its partition class of~$E(m')$, and let $L''$ be an analogous set of leaves in~$H''$. Let $H$ be obtained from $(H'-L')\cup (H''-L'')$ by identifying all the neighbours in~$H'$ of nodes in~$L'$ with all the neighbours in~$H''$ of nodes in~$L''$ into a new vertex~$m$ (Fig.~\ref{fig:contructible}).%
   \COMMENT{}
   Let $M$ be the set of vertices of $H$ that are in $M'\cup M''$ or equal to~$m$, and let $N$ be the set of all other vertices of~$H$ (those in $N'\cup N''$). Let $E'(m)$ consist of the edges at~$m$ that come from~$H'$, and let $E''(m)$  consist of the edges coming from~$H''$.
   \end{enumerate}

  \begin{figure}[htpb]
\centering
   	  \includegraphics
            {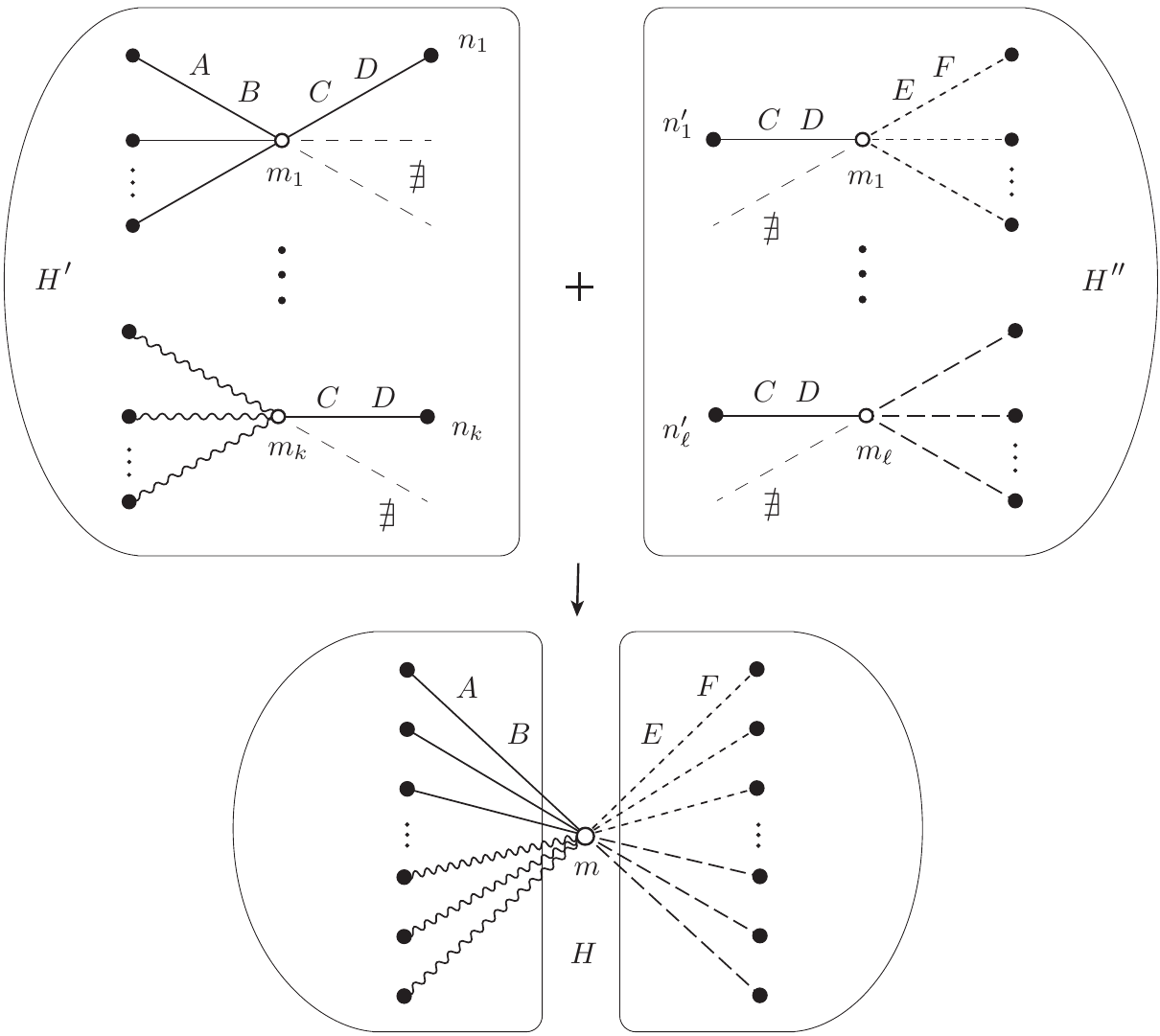}
   	  \caption{The recursion for constructible graphs}
   \label{fig:contructible}
   \end{figure}

\begin{LEM}\label{lem:constructiblecusped}
Constructible graphs are cusped.
\end{LEM}

\begin{proof}
We apply induction following their recursive definition. Since subdivided trees are cusped, the induction starts. Now let $H$ be obtained from two cusped graphs $H', H''$ as in~(P2). Since $m$ is a cutvertex of~$H$ dividing $E'(m)$ and $E''(m)$ into different blocks, no cycle of $H$ through $m$ traverses~$m$. Hence $H$ inherits property (ii) from the definition of `cusped' from the graphs~$H',H''$.

To check property~(i), consider an edge $mn$ of $H$ such that $n$ is a leaf of~$H$.%
   \COMMENT{}
   Assume that $n\in H'$. Then $n$ is a leaf also in~$H'$. Then the partition class of its incident edge $nm'$ in~$H'$, say $E'(m')$, contains only~$nm'$. But by construction of~$H$, the other partition class of edges at~$m$ in~$H'$ also contains only one edge $n'm'$, with $n'$ a leaf of~$H'$. Since $H'$ is connected, this means that $H'$ is just the 2-path $n'm'n$. Hence $L' = \{n'\}$, so $mn$ is the only edge in its partition class also in~$H$.
\end{proof}

We remark that while many cusped graphs are constructible~\cite{pointed}, not all are. For example, a vertex $m\in M$ in a constructible graph will never separate the other ends of two of its incident edges from the same partition class of~$E(m)$, but this can happen in an arbitrary cusped graph such as a tree.%
   \COMMENT{}

\medbreak

We can now construct $S$-graphs in the same way. Let $(H,\alpha)$ be an $S$-graph over $\F\sub 2^S$,%
   \COMMENT{}
   with $V(H) = N\cup M$ and partitions of the sets $E(m)$ as earlier. We say that $(H,\alpha)$ is \emph{constructible} (over~$\F$) if either
   \begin{enumerate}[(S1)]
   \item $(H,\alpha)$ is obtained from an $S$-tree $(T,\alpha')$ over~$\F$ with $T$ a $k$-star on~$N$ ($k\ge 2$) by subdividing every edge once, putting the subdividing vertices in~$M$, and letting $\alpha(n',m) = \alpha(m,n'') = \alpha'(n',n'')$ whenever $m\in M$ subdivides the edge $n'n''\in T$; or
   \item $H$ is obtained as in~(P2) from the disjoint union of two graphs $H',H''$ in $S$-graphs $(H',\alpha')$ and $(H'',\alpha'')$ constructible over~$\F$, in such a way that there exists a separation $(C,D)\in S$ such that
   \begin{itemize}
     \item $\alpha'(m',n') = (C,D) = \alpha''(n'',m'')$ for every $n'\in L'$ with incident edge $n'm'\in H'$ and every $n''\in L''$ with incident edge $m''n''\in H''$;
      \item neither $(C,D)$ nor $(D,C)$ is a leaf separation%
  \footnote{These are separations $\alpha(n,m)$ with $n$ a leaf~\cite{DiestelOumDualityI}.}
       in~$(H,\alpha)$;%
   \COMMENT{}
      \item $\alpha = (\alpha'\!\restriction \vec E(H'-L')) \cup (\alpha''\!\restriction\vec E(H''-L''))$
   \end{itemize}
   \end{enumerate}
 (see Figure~\ref{fig:contructible}). Any $S$-graph $(H,\alpha)$ arising as in~(S2) will be said to have been obtained from $(H',\alpha')$ and $(H'',\alpha'')$ by \emph{amalgamating $L'$ with~$L''$}.

Thus in~(S2), $H$ is obtained from $H'$ and~$H''$ by identifying, for some $(C,D)\in S$, all the leaves of $H'$ with leaf separation $(D,C)$ with all the leaves of $H''$ with leaf separation~$(C,D)$ and contracting all the edges at the identified node into one new amalgamation vertex~$m$.

Checking that $(H,\alpha)$ as obtained in~(S2) is again an $S$-graph is straightforward: conditions (i)--(iii) from the definition of $S$-graphs carry over from $H'$ and~$H''$, while (iv) holds because whenever $\tilde n'm\in E'(m)$ and $m\tilde n''\in E''(m)$ there are $n'\in L'$ and $n''\in L''$ such that
 $$\alpha(\tilde n',m) = \alpha'(\tilde n',m')\ge\alpha'(m',n') = (C,D) = \alpha''(n'',m'')\ge\alpha'' (m'',\tilde n'') = \alpha(m,\tilde n''),$$
where the inequalities hold by (iv) for $(H',\alpha')$ and~$(H''\alpha'')$. (In Figure~\ref{fig:contructible}, for example, we have $(A,B)\ge (E,F)$ in~$(H,\alpha)$, because $(A,B)\ge (C,D)$ in~$(H',\alpha')$ and $(C,D)\ge (E,F)$ in~$(H'',\alpha'')$.)

\section{The General Duality Theorem}\label{sec:Thm}

We can now state and prove the most general version of our duality theorem. It looks for consistent orientations and allows in~$\F$ arbitrary subsets (equivalently by Lemma~\ref{lem:stars}, weak stars) rather than just stars. On the dual side it offers only $S$-graphs rather than $S$-trees as witnesses when such an orientation does not exist,%
   \COMMENT{}
   but we can choose whether we want to use constructible or arbitrary cusped $S$-graphs.%
   \COMMENT{}

\begin{THM}\label{thm:gen}%
   \COMMENT{}
  Let ${S}$ be a finite separation system.  Let ${S}^-\sub S$%
  \COMMENT{}
   and $\F\sub 2^S$. Then the following assertions are equivalent:\vskip-3pt\vskip0pt
\begin{enumerate}[\rm(i)]\itemsep0pt
  \item There exists a consistent $\F$-avoiding orientation of ${S}$ extending ${S}^-$.
  \item There is no cusped ${S}$-graph over $\F$ rooted in ${S}^-$.
  \item There is no ${S}$-graph over $\F$ and rooted in~$S^-$ that is constructible over~$\F$.
  \end{enumerate}
\end{THM}

\noindent
By Lemma~\ref{lem:stars}, we may replace in (i) the set $\F$ with the set%
   \COMMENT{}
   $\F^*$ of weak stars or the set $\F^-$ of minimal weak stars in~$\F^*$ to obtain an equivalent assertion; we may then leave (ii) and (iii) unchanged or change $\F$ there as well, as we wish.

By the same argument, it will not be possible to narrow the class of $S$-graphs allowed in (ii) and~(iii) to any inequivalent%
   \COMMENT{}
   subclass just by restricting $\F$ in this way.%
   \COMMENT{}
  For example, since Theorem~\ref{thm:gen} fails when we replace `$S$-graph' with `$S$-tree' (Example~\ref{ex:P1cycle}), it will still fail with `$S$-tree' when we restrict $\F$ to $\F^*$ or~$\F^-$.\looseness=-1

\begin{proof}[Proof of Theorem~\ref{thm:gen}]
(i)$\to$(ii) Let $O\supe S^-$ be a consistent $\F$-avoiding orientation of~$S$, and suppose there is an $S$-graph $(H,\alpha)$ over~$\F$ rooted in~$S^-$. Let $H$ be given with partitions $V(H) = N\cup M$ and $E(m) = E'(m)\cup E''(m)$ at vertices $m\in M$, as in the definiton of cusped graphs. Now consider the orientation $\vec E(H)$ of $H$ that $O$ induces via~$\alpha$: orient an edge $uv\in E(H)$ from $u$ to~$v$ if $\alpha(u,v)\in O$, and from $v$ to~$u$ if $\alpha(u,v)\notin O$ (and hence $\alpha(v,u)\in O$).%
   \COMMENT{}

   The fact that $O$ extends~$S^-$ while $(H,\alpha)$ is rooted in~$S^-$, the fact that $O$ avoids~$\F$ while $(H,\alpha)$ is over~$\F$,  and the fact that $O$ is consistent while the edges of $H$ at vertices $m\in M$ can witness inconsistency, then imply the following:
\vskip-3pt\vskip0pt
\begin{itemize}\itemsep=0pt
   \item at every node $n\in N$ at least one incident edge is oriented away from~$n$;
   \item at every vertex $m\in M$ either all edges in $E'(m)$ or all edges in $E''(m)$ are oriented away from~$m$.
\end{itemize}
\vskip-3pt\vskip0pt
   This contradicts Lemma~\ref{lem:pointedsink}.

\goodbreak

(ii)$\to$(iii) follows from Lemma~\ref{lem:constructiblecusped}.

(iii)$\to$(i) Suppose first that $S^-$ contains separations $(A,B), (C,D)$%
   \COMMENT{}
  witnessing inconsistency, with $(D,C)\le (A,B)$ say. Let $H$ be a path $nmn'$, and put $N = \{n,n'\}$ and $M=\{m\}$. Let $\alpha(n,m) = (A,B)$ and $\alpha(n',m) = (C,D)$, as well as $\alpha(m,n) = (B,A)$ and $\alpha(m,n') = (D,C)$. Then $(H,\alpha)$ is an $S$-graph as in~(iii),%
   \COMMENT{}
   completing the proof.

We now assume that $S^-$ contains no such $(A,B), (C,D)$. Then $S^-$ is a consistent partial orientation%
   \footnote{A \emph{partial orientation} of $S$ is an orientation of a symmetric subset of~$S$~\cite{DiestelOumDualityI}.}
   of~$S$. We apply induction on $\abs{{S}}-2\abs{{S}^-}$.

  If $\abs{{S}}=2\abs{{S}^-}$, then ${S}^-$ itself is an orientation of~${S}$ extending ${S}^-$. If (i) fails, then $S^-$ has a subset $\{(A_1,B_1),\ldots,(A_n,B_n)\}\in\F$.%
   \COMMENT{}
   Let $N$ be the vertex set of an $n$-star with centre~$t$ and leaves $s_1,\ldots,s_n$. Subdivide every edge $s_i t$ by a new vertex~$s'_i$, and put these in~$M$. Let $\alpha(s_i,s'_i) = \alpha(s'_i,t) = (A_i,B_i)$ and $\alpha(t,s'_i) = \alpha(s'_i,s_i) = (B_i,A_i)$, for $i = 1,\ldots,n$. Then $(H,\alpha)$ is an $S$-graph as in~(iii), completing the proof.%
   \COMMENT{}

  Thus we may assume that ${S}$ has a separation $(X,Y)$ such that neither $(X,Y)$ nor $(Y,X)$ is in ${S}^-$.%
  \COMMENT{}
   Let ${S}_X^-={S}^-\cup\{(Y,X)\}$ and ${S}_Y^-={S}^-\cup\{(X,Y)\}$. Since any orientation of~$S$ extending ${S}_X^-$ or~$S_Y^-$ also extends ${S}^-$, we may assume that there is no such orientation that is both consistent and avoids~$\F$.

By the induction hypothesis, or trivially%
   \COMMENT{}
   if $S_X^-$ or $S_Y^-$ is inconsistent, there are constructible ${S}$-graphs $(H_X,\xi)$ and $(H_Y,\upsilon)$ over~$\F$, rooted in ${S}_X^-$ and $S_Y^-$, respectively. Unless one of these is in fact rooted in~${S}^-$, contradicting~(iii), $H_X$~has at least one leaf~$x$, with neighbour~$x'$ say, such that $\xi(x,x') = (Y,X)$, and $H_Y$ has at least one leaf~$y$, with neighbour~$y'$ say, such that $\upsilon(y,y') = (X,Y)$.\looseness=-1

Let $L_X$ be the set of all these leaves $x$ of~$H_X$, and let $L_Y$ be the set of all these leaves~$y$ of~$H_Y$. It is easily verified that the $S$-graph obtained from $(H_X,\xi)$ and $(H_Y,\upsilon)$ by amalgamating $L_X$ with $L_Y$ is as in~(iii).
   \end{proof}

  \begin{figure}[htpb]
\centering
   	  \includegraphics
              {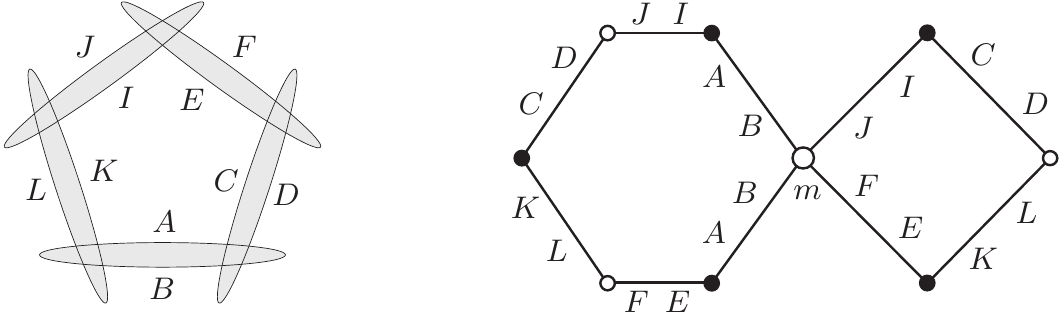}
   	  \caption{An $S$-graph corresponding to a cyclic arrangement of separations}
   \label{fig:5cycleSgraph}
   \end{figure}

Figure~\ref{fig:5cycleSgraph} shows an $S$-graph witnessing the non-existence of a consistent orientation avoiding~$\F$ in Example~\ref{ex:P1cycle} (Fig.~\ref{fig:P1cycle}). It was found as the proof of Theorem~\ref{thm:gen} would suggest: not knowing what to do with $S^-=\emptyset$ we tentatively considered $S_A^- = \{(B,A)\}$ and $S_B^- = \{(A,B)\}$, found $S$-graphs over $\F$ rooted in $S_A^-$ and~$S_B^-$, respectively (in fact, `generalized $S$-paths' as in Section~\ref{sec:obstructions}; Fig.~\ref{fig:Spaths}), and pieced these together to form the rootless $S$-graph shown in Figure~\ref{fig:5cycleSgraph}.

   \begin{figure}[htpb]
\centering
   	  \includegraphics
              {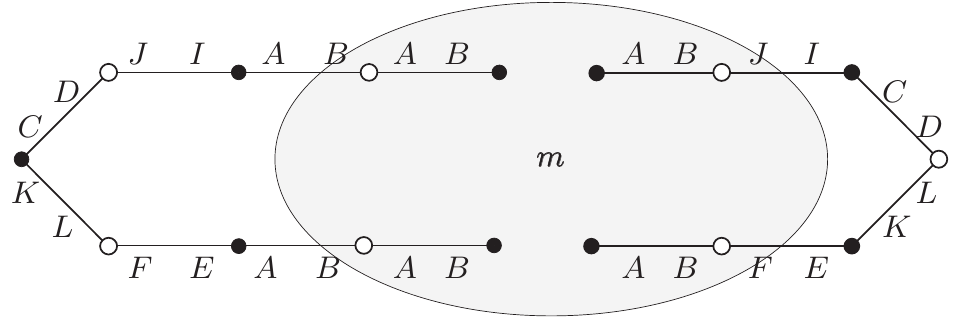}
   	  \caption{The $S$-paths rooted in $S_A^- = \{(B,A)\}$ and $S_B^- = \{(A,B)\}$}
   \label{fig:Spaths}
   \end{figure}

Notice that while it is easy to use this $S$-graph, once found, to prove that $S$ has no consistent $\F$-avoiding orientation, it was not so easy to find it. In general, finding either a consistent $\F$-avoiding orientation of~$S$ or a cusped $S$-graph over~$\F$ is a hard problem: as Bowler~\cite{B14KkandNPC} observed, the problem to decide whether a set $S$ of separations in a graph $G$ has a consistent orientation that avoids a given $\F\sub 2^S$ is NP-hard even when $S$ is fixed as~$S_k$ and the sets in $\F$ consist of only three separations (and $G$ and $\F$ are input).%
   \COMMENT{}

\section{Applications: blocks and profiles}\label{sec:apps}

One of the motivations for this paper was to find a duality theorem for two notions of `dense objects' that received some attention recently, \emph{blocks} and \emph{profiles}.\looseness=-1 

A~\emph{$k$-block} in a graph $G$, where $k$ is any positive integer, is a maximal set $X$ of at least~$k$ vertices such that no two vertices $x,x'\in X$ can be separated in~$G$ by fewer than~$k$ vertices other than $x$ and~$x'$. A~\emph{$k$-profile} of $G$ is a consistent orientation $P$ of the set $S_k$ of separations of order $<k$ in $G$ such that whenever $(A,B),(C,D)\in P$ and $(A\cup C, B\cap D)\in S_k$ then $(A\cup C, B\cap D)\in P$. Every $k$-block induces a $k$-profile (see below on how); tangles of order~$k$ are other examples of $k$-profiles. Every $k$-profile, in turn, is a haven of order~$k$ and thus gives rise to a bramble of order at least~$k$. Thus, profiles lie between blocks and brambles, but also between tangles and havens. See~\cite{confing, ForcingBlocks, profiles, CDHH13CanonicalAlg, CDHH13CanonicalParts} for more on blocks and profiles.

Before we look at blocks and profiles separately, let us note that any consistent orientation $O$ of~$S_k$ extends
 $$ S_k^- := \big\{(A,V) : |A| < k\big\}\sub S_k\,,$$
since $(V,A)\in O$ would violate consistency by $(A,V)\le (V,A)\in O$.%
   \COMMENT{}

Let us look at blocks first. Consider a $k$-block $b$ in a graph~$G$. For every separation $(A,B)\in S_k$ we have $b\sub A$ or $b\sub B$, but not both (since $|b|\ge k$). Thus,
 $$O_k(b) := \{\,(A,B)\in S_k : b \sub B\,\}$$
 is a consistent orientation of~$S_k$. Clearly, $O_k(b)$~has no subset in 
 \begin{equation}\label{eq_Bk}
\textstyle \B_k := \big\{\,\{ (A_i,B_i)\mid i=1,\dots,n\}\sub S_k : \big|\bigcap_{i=1}^n B_i\big| < k\,\big\},
\end{equation}
 so for $\F = \B_k$ it is an $\F$-avoiding consistent orientation of~$S_k$.

Conversely, every $\B_k$-avoiding orientation $O$ of~$S_k$ is clearly consistent,%
   \COMMENT{}
   and
 $$\textstyle b_k(O) := \bigcap \{\,B : (A,B)\in O\,\}$$
 is a $k$-block: no separation in $S_k$ separates any of its vertices, since that separation or its inverse lies in~$O$ and hence has a side containing~$b_k(O)$;%
   \COMMENT{}
   and since $O\notin\B_k$ because $O$ avoids~$\B_k$, we have $|b_k (O)|\ge k$ by~\eqref{eq_Bk}.

We can thus obtain orientations of $S_k$ from $k$-blocks, and vice versa. These operations are inverse to each other:

\begin{LEM}\label{lem:blocks}
$O = O_k(b)$ if and only if $b = b_k(O)$.
\end{LEM}

\begin{proof}
Assume first that $O = O_k(b)$.%
   \COMMENT{}
   Then $O$ consists of all the separations $(A,B)\in S_k$ such that $b\sub B$. So the intersection of all those~$B$ contains~$b$ too, and this intersection is $b_k(O)$. Thus, $b\sub b_k(O)$.

Conversely, since $b$ is maximal as an $S_k$-inseparable set of vertices, any vertex $v\notin b$ is separated from some vertex of~$b$ by a separation in~$S_k$, and hence also by some separation $(A,B)\in O$.%
   \COMMENT{}
   Since $b\sub B$ by definition of $O=O_k(b)$, this means that $v\in A\sm B$. So $v$ lies outside~$B$, and hence also outside the intersection $b_k(O)$ of all $B$ with $(A,B)\in O$. Thus, $b_k(O)\sub b$.

Assume now that $b = b_k(O)$.%
   \COMMENT{}
   For every $(A,B)\in O$, its `large side' $B$ trivially contains the intersection of all such `large sides' of separations in~$O$. So $B\supe b_k(O) = b$ and hence $(A,B)\in O_k(b)$, proving $O\sub O_k(b)$.

Conversely, if $(A,B)\in O_k(b)$, then $b_k(O) = b\sub B$. Since $b_k(O)$ is a $k$-block%
  \COMMENT{}
   and hence $|b_k(O)|\ge k$, this means that $b_k(O)\not\sub A$. Since $(B,A)\in O$ would imply $b_k(O)\sub A$, by definition of~$b_k(O)$, we thus have $(B,A)\notin O$ and hence $(A,B)\in O$, as desired.
\end{proof}

Thus, $k$-blocks `are' precisely the consistent $\B_k$-avoiding orientations of~$S_k$. When we now treat $k$-blocks in our duality framework, we shall use the term `$k$-block' also to refer to these orientations.

Rephrasing $k$-blocks as orientations of~$S_k$ throws up an interesting connection between blocks and brambles or havens that had not been noticed before. As $\B_k$ is closed under taking subsets, it contains the set $\B_k^*$ defined in Section~\ref{sec:stars}:
 $$\B_k^* = \{\, S\in\B_k : S\text{ is a weak star}\}$$
 Hence by Proposition~\ref{lem:stars} and Lemma~\ref{lem:blocks}, the $k$-blocks of $G$ are precisely the consistent orientations of $S_k$ that contain no weak star from~$\B_k$. If we delete the word `weak' in this sentence, we obtain the consistent orientations of~$S_k$ that avoid
 $$\S_k^* = \{\, S\in\B_k : S\text{ is a star}\},$$
 which are precisely the dual objects to \td s of width~$<k-1$.%
   \footnote{As shown in~\cite{DiestelOumDualityI}, $G$ has a haven of order at least~$k$, or equivalently a bramble of order at least~$k$, if and only if $S_k$ has a consistent $S_k^*$-avoiding orientation.}

Our General Duality Theorem thus specializes to blocks as follows:

\begin{THM}\label{thm:blocks}
For every finite graph $G$ and $k>0$ the following statements are equivalent:
\vskip-3pt\vskip0pt
\begin{enumerate}[\rm(i)]\itemsep=0pt
 \item $G$ contains a $k$-block.
 \item $S_k$ has a $\B_k$-avoiding orientation (which is consistent and extends~$S_k^-$).
 \item $S_k$ has a consistent $\B_k^*$-avoiding orientation (which extends~$S_k^-$).
 \item There is no $S_k$-graph over~$\B_k$ rooted in~$S_k^-$.
 \item There is no $S_k$-graph over~$\B_k^*$ rooted in~$S_k^-$.\qed
\end{enumerate}
\end{THM}

\noindent
   We remark that, by Theorem~\ref{thm:gen}, the $S_k$-graphs in (iv) and~(v) can be chosen to be constructible over~$\B_k$ or~$\B_k^*$, respectively.

\medbreak

To wind up our treatment of blocks, let us mention a quantitative%
   \footnote{rather than structural, as our Theorem~\ref{thm:blocks}}
   duality theorem for blocks proved in~\cite{ForcingBlocks}. There, the \emph{block-width} ${\rm bw}(G)$ of~$G$ is defined as the least integer~$\ell$ such that $G$ can be divided recursively into parts of size at most~$\ell$ by separations of~$G$%
   \COMMENT{}
   of order at most~$\ell$. (See~\cite[Section~7]{ForcingBlocks} for details.) It is not hard to show that this can be done if and only if $G$ has no $k$-block for $k>\ell$. The \emph{block number}
 $$\beta(G) := \max\,\{\,k : G\text{ has a $k$-block}\,\}$$
thus comes with the following quantitative duality~\cite{ForcingBlocks}:

\begin{PROP}
Every finite graph $G$ satisfies $\beta(G) = {\rm bw}(G)$.\qed
\end{PROP}

Let us now turn to profiles. A~\emph{$k$-profile} of $G$ is a consistent orientation $P$ of $S_k$ satisfying
  $$(A,B), (C,D)\in P\ \Rightarrow\ (B\cap D, A\cup C)\notin P.\eqno(\rm P)$$
 This indirect definition is convenient, because it is not clear whether the separation $(A\cup C, B\cap D)$ is in~$S_k$, i.e., has order~$<k$: if it does, it must be in~$P$, but if it does not, then there is no requirement. So the $k$-profiles of $G$ are precisely the consistent orientations of~$S_k$ that avoid
 $$\P_k = \big\{\,\{(A,B),(C,D),(E,F)\}\sub S_k : (E,F) = (B\cap D,A\cup C)\big\}. $$
Let $\P_k^*$ and $\P_k^-$ be obtained from $\P_k$ as defined in Section~\ref{sec:stars}. While crossing sets $S\in\P_k$ will be weak stars, some nested ones are not; in particular, $S$~may be inconsistent.%
   \COMMENT{}
   However, if $ S=\{(A,B),(C,D),(E,F)\}$ with $(A,B)\le (C,D)$, then $S^+ = \{(C,D),(C,D),(D,C)\}\in\P_k$, so $\P_k^*\sub\P_k$.

Our General Duality Theorem specializes to profiles as follows:

\begin{THM}\label{thm:profiles}
For every finite graph $G$ and $k>0$ the following statements are equivalent:
\vskip-3pt\vskip0pt
\begin{enumerate}[\rm(i)]\itemsep=0pt
 \item $G$ contains a $k$-profile.
 \item $S_k$ has a consistent $\P_k$-avoiding orientation (which extends~$S_k^-$).
 \item There is no $S_k$-graph over~$\P_k$ rooted in~$S_k^-$.\qed
\end{enumerate}
\end{THM}

\goodbreak

\noindent
   As before, we can replace $\P_k$ with $\P_k^*$ or~$\P_k^-$ {\sl ad libitum},%
   \COMMENT{}
   and the $S_k$-graph in~(iii) can be chosen to be constructible over $\P_k$, $\P_k^*$ or~$\P_k^-$ as desired.

\medbreak

As in the case of $k$-blocks, one may be tempted to compare the consistent $\P_k$-avoiding orientations of $S_k$ with those that are required only to avoid the (proper) stars in~$\P_k$. If these are still just the $k$-profiles, we can use the Strong Duality Theorem from~\cite{DiestelOumDualityI} to characterize them by proper $S$-trees rather than just $S$-graphs, improving Theorem~\ref{thm:profiles}.

So, given a triple $\{(A,B),(C,D),(E,F)\}\in\P_k$ with $(A,B),(C,D)$ crossing and $(E,F) = (B\cap D,A\cup C)$, we can `uncross' it to obtain the nested triples
 $$(A\cap D,B\cup C), (C,D), (E,F) \quad\text{and}\quad (A,B), (C\cap B, A\cup D), (E,F).$$
 As is easy to check, at least one of $(A\cap D,B\cup C)$ and $(C\cap B, A\cup D)$ must again be in~$S_k$ (by the `submodularity' of the order of graph separations), and it would be natural to expect that the corresponding triple might then be in~$\P_k$. If that was always the case, then every orientation of~$S_k$ avoiding~$\P_k^*$ would in fact avoid all of~$\P_k$. The $k$-profiles of $G$ would then be precisely the consistent orientations of $S_k$ not containing a star in~$\P_k$.

Unfortunately, however, the triple $(A\cap D,B\cup C), (C,D), (E,F)$ (say) can fail to lie in~$\P_k$ for a different reason: even if the separation $(A',B') = (A\cap D,B\cup C)$ replacing $(A,B)$ lies in~$S_k$, it can happen that $(E,F)\ne (A'\cup C,B'\cap D)$. We leave the details to the reader to check.%
   \COMMENT{}

\medbreak

In the remainder of this section we shall see that this is not just an obstacle that might be overcome: $k$-profiles are indeed not in general dual to $S_k$-trees over stars in~$\P_k$, and $k$-blocks are not in general dual to $S_k$-trees over stars in~$\B_k$ (all over~$S_k^-$).%
   \footnote{As noted earlier, the latter also follows indirectly from the fact that $k$-blocks are not the same as $k$-havens, which \emph{are} dual to the $S_k$-trees over stars in~$\B_k$~\cite{DiestelOumDualityI}. The $S_k^-$ considered in~\cite{DiestelOumDualityI} was slightly larger, containing also the proper separations $(A,B)$ with $|A|<k$, but $S_k$-trees over~$\B_k^*$ rooted in these can easily be extended to $S_k$-trees over~$\B_k^*$ rooted in our~$S_k^-$.}%
   \COMMENT{}

  \begin{figure}[htpb]
\centering
   	  \includegraphics[width=\textwidth]
            {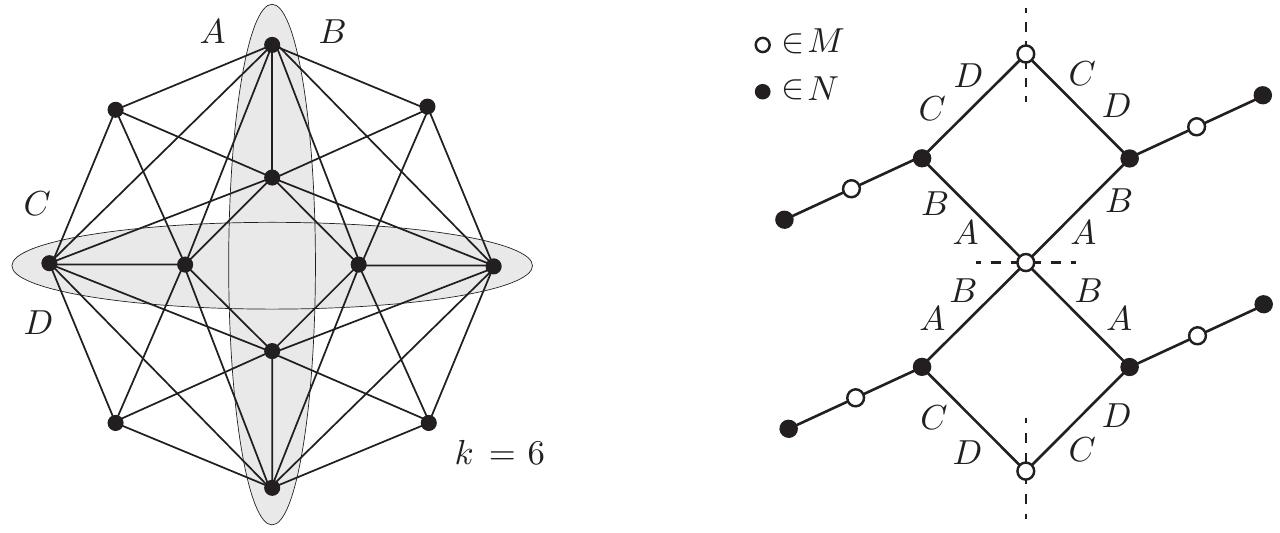}
   	  \caption{The graph of Example~\ref{ex:JC} and its $S$-graph}
   \label{fig:JC}
   \end{figure}

\begin{EX}\label{ex:JC}\rm\sloppy
Consider the graph $G$ on a set $V$ of 12 vertices shown in Figure~\ref{fig:JC}. It has five $K_5$s separated in pairs by two crossing separations $(A,B),(C,D)\in S_5$. The only other proper separations are of the form ${(E,F)\in S_5}$ (or its inverse) with $E$ spanning one of $K_5$s and $F=V\sm\{v\}$ for the `corner vertex'~$v$ of that~$K_5$.

Note that every pair of non-adjacent vertices is separated by one of these 4-separations. Hence $G$ has four 5-blocks, the $K_5$s, but no 6-block. It also has no $6$-profile.%
   \COMMENT{}
   Indeed, if it did then by symmetry we could assume that this profile~$P$ contains $(A,B)$ and~$(C,D)$. By condition~(P), it then also contains the 4-separation $(F,E)$ with $F= A\cup C$ and $E= (B\cap D)$ spanning the bottom-right~$K_5$ in Figure~\ref{fig:JC}. (We call this a \emph{corner separation}.) As $|E|<6$, we also have $(E,V)\in S_6^-\sub P$. Applying~(P) to $(F,E)$ and~$(E,V)$ we obtain $(E,V) = (E\cap V, F\cup E)\notin P$,%
   \COMMENT{}
    a contradiction.

However, $G$ has no `generalized $S$-tree'~-- i.e., no $S$-graph $(H,\alpha)$ with $H$ a tree~-- over either $\B_6$ or~$\P_6$ and rooted in~$S_6^-$. Indeed, by examining $\B_6$ and~$\P_6$ it is not hard to show that in any such $(H,\alpha)$ at least two edges at any interior node must map to one of the separations $(A,B), (C,D)$ or its inverse, so these edges will lead to another interior node.%
   \COMMENT{}
   We can thus find either a cycle or an infinite paths in~$H$, a contradiction.%
   \COMMENT{}
\end{EX}

Figure~\ref{fig:JC} shows an $S_6$-graph of $G$ over $\P_6\sub\B_6$, rooted in~$S_6^-$. The unlabelled edges at the leaves map to the `corner separations' $(E,F)$ mentioned earlier, with $E$ spanning a~$K_5$.

However, it is still conceivable (although we consider it unlikely) that $k$-blocks and $k$-profiles are dual to $S_k$-trees over some other collection~$\F$, even of stars, in~$S_k$. For example, for profiles we might consider as $\F$ the 3-stars of separations in~$S_k$ that are the result of `uncrossing' a triple violating~(P):
  $$\F := \Big\{ \{(A,B),(C,D),(E,F)\}\sub S_k\ \Big|\ \exists\, (A',B')\in S_k :
     {{(A,B) = (A'\cap D, B'\cup C)}\atop {(E,F) = (B'\cap D, A'\cup C)}}\Big\} $$
(Fig.~\ref{fig:uncross}).
  \begin{figure}[htpb]
\centering
   	  \includegraphics
             {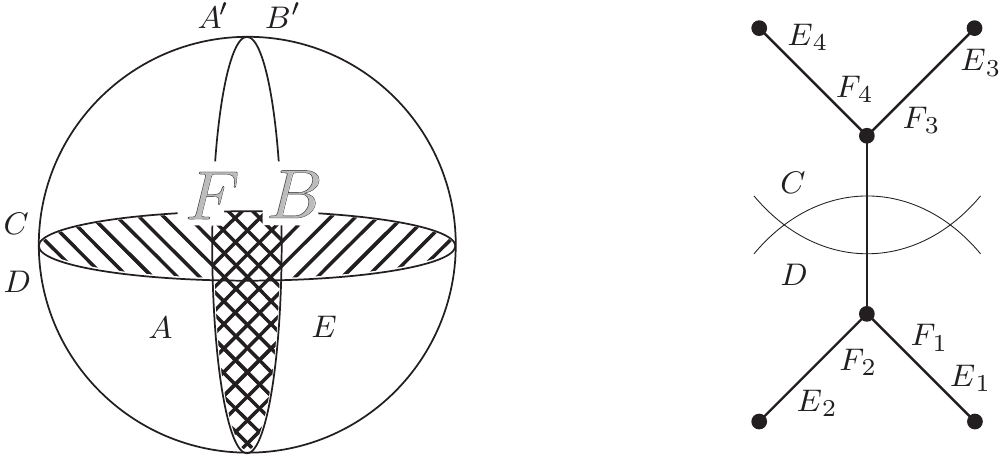}
   	  \caption{Uncrossing a weak star violating~(P), and an $S_k$-tree over~$\F$}
   \label{fig:uncross}
   \end{figure}
The graph $G$ of Example~\ref{ex:JC} has an $S_k$-tree $(T,\alpha)$ over this~$\F$, in which $T$ has four leaves corresponding to the four corner separations and two internal nodes joined by an edge whose orientations map to $(C,D)$ and~$(D,C)$.

\begin{PROBLEM}\label{problem:uncross}
Is $S_k$ $\F$-separable in the sense of~{\rm\cite{DiestelOumDualityI}}?
\end{PROBLEM}

\noindent
  A positive answer would make \cite[Theorem~4.4]{DiestelOumDualityI} applicable to $S_k$ with this~$\F$. The submodularity of~$S_k$ implies that every triple of separations in $S_k$ that violates~(P) can be `uncrossed', in one of two potential ways, to yield a triple in~$\F$. Hence the $k$-profiles are also precisely the $\F$-avoiding consistent orientations of~$S_k$.%
   \COMMENT{}
   A~positive answer to Problem~\ref{problem:uncross} would thus yield a `strong' duality theorem~-- one with $S_k$-trees over stars~-- for $k$-profiles after all.

\bibliographystyle{abbrv}
\bibliography{collective}

\begin{thebibliography}{1}

\bibitem{B14KkandNPC}
N.~Bowler.
\newblock Presentation at \emph{Hamburg workshop on graphs and matroids},
  Spiekeroog~2014.

\bibitem{CDHH13CanonicalAlg}
J.~Carmesin, R.~Diestel, M.~Hamann, and F.~Hundertmark.
\newblock Canonical tree-decompositions of finite graphs {I.~Existence} and
  algorithms.
\newblock {\em J. Combin. Theory Ser. B}, to appear.

\bibitem{CDHH13CanonicalParts}
J.~Carmesin, R.~Diestel, M.~Hamann, and F.~Hundertmark.
\newblock Canonical tree-decompositions of finite graphs {II.~The} parts.
\newblock Preprint 2013.

\bibitem{ForcingBlocks}
J.~Carmesin, R.~Diestel, M.~Hamann, and F.~Hundertmark.
\newblock $k$-{B}locks: a connectivity invariant for graphs.
\newblock Preprint 2013.

\bibitem{confing}
J.~Carmesin, R.~Diestel, F.~Hundertmark, and M.~Stein.
\newblock Connectivity and tree structure in finite graphs.
\newblock {\em Combinatorica}, 34(1):1--35, 2014.

\bibitem{DiestelOumDualityI}
R.~Diestel and S.~Oum.
\newblock Unifying duality theorems for width parameters, {I.~W}eak and strong
  duality.
\newblock Preprint 2014.

\bibitem{profiles}
F.~Hundertmark.
\newblock Profiles. {A}n algebraic approach to combinatorial connectivity.
\newblock arXiv:1110.6207, 2011.

\bibitem{pointed}
D.~Shoesmith and T.~Smiley.
\newblock Theorem on directed graphs, applicable to logic.
\newblock {\em J. Graph Theory}, 3(4):401--406, 1979.

\end{thebibliography}
\end{document}